\documentclass[12pt]{amsart}

\usepackage[utf8]{inputenc}
\usepackage[english]{babel}

\usepackage[letterpaper,top=2cm,bottom=2cm,left=3cm,right=3cm,marginparwidth=1.75cm]{geometry}

\usepackage{amsmath}
\usepackage{graphicx}
\usepackage[colorlinks=true, allcolors=blue]{hyperref}

\usepackage{amssymb} 
\usepackage{comment}
\usepackage{indentfirst}
\usepackage{csquotes}
\usepackage{theoremref}
\usepackage{hyperref}
\usepackage{geometry}

\hypersetup{hidelinks, colorlinks=true, allcolors=cyan, pdfstartview=Fit, breaklinks=true}

\usepackage{amsthm}

\theoremstyle{plain}
\newtheorem{theorem}{Theorem}
\numberwithin{theorem}{section}

\newtheorem*{corollary*}{Corollary}
\newtheorem*{Example*}{Example}

\newtheorem{conjecture}[theorem]{Conjecture}
\theoremstyle{definition}

\newtheorem*{def*}{Definition}
\newtheorem*{theorem*}{Theorem}

\newtheorem*{examples}{Examples}
\newtheorem*{definition*}{Definition}

\theoremstyle{remark}
\newtheorem*{remark}{Remark}

\newcommand{\bracket}[1]{\left( #1 \right)}
\newcommand{\modulo}[3]{#1\equiv#2\ \bracket{\mathrm{mod}\ #3}}

\numberwithin{equation}{section}

\title{\textbf{Divisibility and distribution of 5-regular partitions}}
\author{QI-YANG ZHENG}
\date{} 

\address{Department of Mathematics, Sun Yat-sen University(Zhuhai Campus), Zhuhai}
\email{zhengqy29@mail2.sysu.edu.cn}

\begin{document}
\maketitle

\begin{abstract}
In this paper we study $b_5(n)$, the $5$-regular partitions of $n$. Using the theory of modular forms, we prove several theorems on the divisibility and distribution properties of $b_5(n)$ modulo prime $m\geq5$. In particular, we prove that there are infinitely many Ramanujan-type congruences modulo prime $m\geq5$.
\end{abstract}


\section{Introduction and statement of results}

\subsection{Introduction}

In number theory, we usually denote $p(n)$ as the number of the partitions of $n$. Ramanujan found the three remarkable congruences as follows:
\begin{equation}
    \notag
    \begin{aligned}
    p(5n+4)&\equiv0\ (\mathrm{mod}\ 5),\\
    p(7n+5)&\equiv0\ (\mathrm{mod}\ 7),\\
    p(11n+6)&\equiv0\ (\mathrm{mod}\ 11).
    \end{aligned}
\end{equation}

\noindent
Such congruences are called Ramanujan-type congruences.

~

For $k\in\mathbb{Z}_{>1}$, we define the $k$-regular partitions $b_k(n)$ by
$$\sum_{n=0}^\infty b_k(n)q^n=\prod_{n=1}^\infty\frac{1-q^{kn}}{1-q^n}.$$

\noindent
We study the arithmetical properties of $b_5(n)$ in this paper. Hirschhorn and Sellers\cite{hirschhorn2010elementary} prove that there are infinitely many Ramanujan-type congruences of $b_5(n)$ modulo $2$. Gordon and Ono\cite{gordon1997divisibility} prove that
$$b_5(5n+4)\equiv0\ (\mathrm{mod}\ 5).$$

Up to now, Ramanujan-type congruence of $b_5(n)$ modulo prime $m\geq7$ has not been found. Our main result is that for each prime $m\geq5$, there exist infinitely many Ramanujan-type congruences modulo $m$. For example, we obtain
$$b_5(2023n+99)\equiv0\ (\mathrm{mod}\ 7)$$

\noindent
satisfied for each nonnegative integer $n$. We will give more examples in \textbf{Section} \ref{examples}.

\subsection{Statement of results}

\begin{theorem}
\label{main theorem}
Let $m\geq5$ be a prime. Then a positive density of primes $l$ have the property that
\begin{equation}
    b_5\left( \frac{mln-1}{6} \right)\equiv0\ (\mathrm{mod}\ m)
\end{equation}

\noindent
satisfied for each integer $n$ with $(n,l)=1$.
\end{theorem}

~

The theorem immediately implies that there are infinitely many Ramanujan-type congruences of $b_5(n)$ modulo $m$. Moreover, together with the Chinese Remainder Theorem, we obtain that if $m$ is a squarefree integer coprime to $3$, then there are infinitely many Ramanujan-type congruences of $b_5(n)$ modulo $m$.

Surprisingly, we do not know whether there is Ramanujan-type congruence of $b_5(n)$ modulo $3$.

Theorem \ref{main theorem} also implies that
$$\#\{ 0\leq n\leq X\ |\ b_5(n)\equiv0\ (\mathrm{mod}\ m) \}\gg X,$$

\noindent
where $m\geq5$ is a prime. For other residue classes $i\not\equiv0\ (\mathrm{mod}\ m)$, we also provide a useful criterion to obtain similar result.

\begin{theorem}
\label{other residue classes}
Let $m\geq5$ be a prime. If there exists one $k\in\mathbb{Z}$ such that
$$b_5\left( mk+\frac{m^2-1}{6} \right)\equiv e\not\equiv0\ (\mathrm{mod}\ m),$$

\noindent
then for each $i=1,2,\cdots,m-1$, we have
$$\#\{ 0\leq n\leq X\ |\ b_5(n)\equiv i\ (\mathrm{mod}\ m) \}\gg \frac{X}{\log X}.$$

\noindent
Moreover, if such $k$ exists, then $k<10(m-1)$.
\end{theorem}

~

\noindent
The congruence of Gordon and Ono show that our criterion is inapplicable for the case $m=5$.

\section{Notation and definitions}

Our proof is depending on the theory of modular forms. First recall that the Dedekind's eta function is defined by:
$$\eta(z)=q^{\frac{1}{24}}\prod_{n=1}^\infty(1-q^n),$$

\noindent
where $q=e^{2\pi iz}$. It is well-known that $\eta(z)$ is holomorphic and never vanishes in the upper half plane.

Now we introduce the $U$ operator. If $j$ is a positive integer, we define $U(j)$ as follows:
\begin{equation}
    \left( \sum_{n=0}^\infty a(n)q^n \right)\ |\ U(j) = \sum_{n=0}^\infty a(jn)q^n.
\end{equation}

\noindent
Sometimes the following expression of $U(j)$ operator is more convenient for computation.
\begin{equation}
    \left( \sum_{n=0}^\infty a(n)q^n \right)\ |\ U(j) = \sum_{\genfrac{}{}{0pt}{}{n=0}{j|n}}^\infty a(n)q^{\frac{n}{j}}.
\end{equation}

\noindent
We define $M_k(\Gamma_0(N),\chi)_m$ as the reduction mod $m$ of the $q$-expansions of modular forms in $M_k(\Gamma_0(N),\chi)$ with integral coefficients. Moreover, we define $S_k(\Gamma_0(N),\chi)_m$ in a similar way.

\section{Proof of Theorem \ref{main theorem}}

Before proving Theorem \ref{main theorem}, we list some useful results. The following theorem is due to Gordon and Hughes\cite{gordon1993multiplicative}.

\begin{theorem}[B. Gordon, K. Hughes]
\label{eta-quotient}
Let $f(z)=\prod_{\delta|N}\eta(\delta z)^{r_\delta}$ be an $\eta$-quotient for which
\begin{equation}
    \notag
    \sum_{\delta|N}\delta r_\delta\equiv0\ (\mathrm{mod}\ 24),
\end{equation}
\begin{equation}
    \notag
    \sum_{\delta|N}\frac N\delta r_\delta\equiv0\ (\mathrm{mod}\ 24),
\end{equation}
\begin{equation}
    \notag
    k:=\frac12\sum_{\delta|N} r_\delta\in\mathbb{Z},
\end{equation}
\noindent
then $f(z)$ satisfies
\begin{equation}
    \notag
    f\left( \frac{az+b}{cz+d} \right)=\chi(d)(cz+d)^kf(z)
\end{equation}
\noindent
for each $\begin{pmatrix}
 a & b\\
 c & d
\end{pmatrix}\in\Gamma_0(N)$. Here $\chi$ is a Dirichlet character modulo $N$ and
\begin{equation}
    \notag
    \chi(n)=\left( \frac{(-1)^k\prod_{\delta|N}\delta^{r_\delta}}{n} \right),
\end{equation}
for each positive odd number $n$.
\end{theorem}

~

Though Theorem \ref{eta-quotient} ensures that $f(z)$ is weakly modular, we need to show that the order at the cusps of $\Gamma_0(N)$ are nonnegative(resp. positive) to obtain that $f(z)$ is a modular(resp. cusp) form. The following theorem of Martin\cite{martin1996multiplicative, ono2004web} provides the explicit expression of the order at cusps.

\begin{theorem}[Y. Martin]
\label{order of cusps}
Let $c,d,N$ be positive integers with $d|N$ and $(c,d)=1$ and $f(z)$ is an $\eta$-quotient satisfying the conditions of Theorem \ref{eta-quotient}, then the order of vanishing of $f(z)$ at the cusp $\frac{c}{d}$ is
\begin{equation}
    \notag
    \frac{N}{24}\sum_{\delta|N}\frac{r_\delta(d^2,\delta^2)}{\delta(d^2,N)}.
\end{equation}
\end{theorem}

~

The following theorem due to Serre\cite{serre1974divisibilite} show that cusp forms have very nice arithmetic properties.

\begin{theorem}[J.-P. Serre]
\label{Serre's theorem}
The set of primes $l\equiv-1\ (\mathrm{mod}\ Nm)$ such that
$$f\ |\ T(l)\equiv0\ (\mathrm{mod}\ m)$$

\noindent
for each $f(z)\in S_k(\Gamma_0(N),\psi)_m$ has positive density, where $T(l)$ denotes the usual Hecke operator acting on $S_k(\Gamma_0(N),\psi)$.
\end{theorem}

~

\begin{proof}[Proof of Theorem \ref{main theorem}]
For a fixed prime $m$, let
\begin{equation}
    f(m;z):=\frac{\eta(5z)}{\eta(z)}\eta^a(5mz)\eta^b(mz),
\end{equation}
\noindent
where $m':=(m\ \mathrm{mod}\ 6)$ and $a:=5-m',\ b:=m'-1$. It is easy to show that $\modulo{f(m;z)}{\eta^{am+1}(5z)\eta^{bm-1}(z)}{m}$ and
$$\eta^{am+1}(5z)\eta^{bm-1}(z)\in S_{2m}(\Gamma_0(5),\chi_5),$$

\noindent
where $\chi_5(n)=\bracket{\frac{n}{5}}$. On the other hand,
\begin{equation}
    f(m;z)=\sum_{n=0}^\infty b_5(n)q^{\frac{24n+m(5a+b)+4}{24}}\cdot\prod_{n=1}^\infty(1-q^{5mn})^a(1-q^{mn})^b.
\end{equation}
\noindent
Acting the $U(m)$ operator on $f(z)$ and since $\modulo{U(m)}{T(m)}{m}$, obtaining
\begin{equation}
    \label{after U/T}
    \modulo{\sum_{n=0}^\infty b_5(n)q^{\frac{24n+m(5a+b)+4}{24}}\ |\ U(m)}{\frac{\eta^{am+1}(5z)\eta^{bm-1}(z)\ |\ T(m)}{\prod_{n=1}^\infty(1-q^{5n})^a(1-q^{n})^b}}{m},
\end{equation}

\noindent
where $T(m)$ denotes usual Hecke operator acting on $S_{2m}(\Gamma_0(5),\chi_5)$. As for the LHS of (\ref{after U/T}), we have
\begin{equation}
    \sum_{n=0}^\infty b_5(n)q^{\frac{24n+m(5a+b)+4}{24}}\ |\ U(m)=\sum_{\genfrac{}{}{0pt}{}{n=0}{m|6n+1}}^\infty b_5(n)q^{\frac{24n+m(5a+b)+4}{24m}}.
\end{equation}

\noindent
Using Theorem \ref{eta-quotient} and \ref{order of cusps}, one can verify that $\eta^4(5z)\eta^4(z)\in S_4(\Gamma_0(5))$ and have the order of $1$ at all cusps. Thus we can write $\eta^{am+1}(5z)\eta^{bm-1}(z)\ |\ T(m)=\eta^4(5z)\eta^4(z)g(m;z)$, where $g(m;z)\in M_{2m-4}(\Gamma_0(5),\chi_5)_m$. Hence
\begin{equation}
    \modulo{\sum_{\genfrac{}{}{0pt}{}{n=0}{m|6n+1}}^\infty b_5(n)q^{\frac{6n+1}{6m}}}{\eta^{4-a}(5z)\eta^{4-b}(z)g(m;z)}{m}.
\end{equation}

\noindent
Replacing $q$ by $q^6$ shows that
\begin{equation}
    \modulo{\sum_{\genfrac{}{}{0pt}{}{n=0}{m|6n+1}}^\infty b_5(n)q^{\frac{6n+1}{m}}}{\eta^{4-a}(30z)\eta^{4-b}(6z)g(m;6z)}{m}.
\end{equation}

\noindent
Since $b_5(n)$ vanishes for non-integer $n$, so
\begin{equation}
    \modulo{\sum_{n=0}^\infty b_5\bracket{\frac{mn-1}{6}}q^{n}}{\eta^{4-a}(30z)\eta^{4-b}(6z)g(m;6z)}{m}.
\end{equation}

\noindent
Moreover, one can verify that $\eta^{4-a}(30z)\eta^{4-b}(6z)\in S_2(\Gamma_0(180))$. Let
\begin{equation}
    \sum_{n=0}^\infty a(n)q^n=\eta^{4-a}(30z)\eta^{4-b}(6z)g(m;6z)\in S_{2m-2}(\Gamma_0(180),\chi_5).
\end{equation}

\noindent
By Theorem \ref{Serre's theorem}, the set of primes $l$ such that
$$\sum_{n=0}^\infty a(n)q^n\ |\ T(l)\equiv0\ (\mathrm{mod}\ m)$$

\noindent
has positive density, where $T(l)$ denotes Hecke operator acting on $S_2(\Gamma_0(180),\chi_5)$. Moreover, by the theory of Hecke operator, we have
\begin{equation}
    \sum_{n=0}^\infty a(n)q^n\ |\ T(l)=\sum_{n=0}^\infty \bracket{a(ln)+\bracket{\frac{l}{5}}l^{2m-3}a\bracket{\frac{n}{l}}}q^n.
\end{equation}

Since $a(n)$ vanishes for non-integer $n$, $a(n/l)=0$ when $(n,l)=1$. Thus $\modulo{a(ln)}{0}{m}$ when $(n,l)=1$. Recalling that $\modulo{a(n)}{b_5\bracket{\frac{mn-1}{6}}}{m}$, we obtain
\begin{equation}
    \modulo{b_5\bracket{\frac{mln-1}{6}}}{0}{m}
\end{equation}

\noindent
satisfied for each integer $n$ with $(n,l)=1$.

\end{proof}

\section{Proof of Theorem \ref{other residue classes}}

First we recall another important theorem of Serre\cite{serre1974divisibilite}.

\begin{theorem}[J.-P. Serre]
\label{Serre's theorem v2}
The set of primes $l\equiv1\ (\mathrm{mod}\ Nm)$ such that
$$a(nl^r)\equiv(r+1)a(n)\ (\mathrm{mod}\ m)$$

\noindent
for each $f(z)=\sum_{n=0}^\infty a(n)q^n\in S_k(\Gamma_0(N),\psi)_m$ has positive density, where $r$ is a positive integer and $n$ is coprime to $l$.
\end{theorem}

~

Here we introduce a theorem of Sturm, which provide a useful criterion to get some congruences via finite computation. Variants of Sturm's Theorem are stated in \cite{murty1997congruences, ono2004web, sturm1987congruence}.

\begin{theorem}[J. Sturm]
\label{Sturm's theorem}
Suppose $f(z)=\sum_{n=0}^\infty a(n)q^n\in M_k(\Gamma_0(N),\chi)_m$ such that
$$a(n)\equiv0\ (\mathrm{mod}\ m)$$

\noindent
for all $n\leq \frac{kN}{12}\prod_{p|N}\left( 1+\frac1p \right)$. Then $a(n)\equiv0\ (\mathrm{mod}\ m)$ for all $n\in\mathbb{Z}$.
\end{theorem}

~

\begin{proof}[Proof of Theorem \ref{other residue classes}]
Let $m\geq5$ be a prime. Suppose $k\in\mathbb{Z}$ such that
$$b_5\left( mk+\frac{m^2-1}{6} \right)\equiv e\not\equiv0\ (\mathrm{mod}\ m),$$

\noindent
let $s=6k+m$. Since $b_5(n)$ vanishes for negative $n$, we have $mk+(m^2-1)/6\geq0$. Hence $s=6k+m>0$ and
$$b_5\left(\frac{ms-1}{6}\right)=b_5\left( mk+\frac{m^2-1}{6} \right)\equiv e\ (\mathrm{mod}\ m).$$

\noindent
For a fix prime $m\geq5$, let $S(m)$ denote the set of primes $l$ such that
$$a(nl^r)\equiv(r+1)a(n)\ (\mathrm{mod}\ m)$$

\noindent
for each $f(z)=\sum_{n=0}^\infty a(n)q^n\in S_{2m-2}(\Gamma_0(180),\chi_{5})_m$, where $r$ is a positive integer and $(n,l)=1$. Recalling that $\sum_{n=0}^\infty b_5\left(\frac{mn-1}{6}\right)q^n\in S_{2m-2}(\Gamma_0(180),\chi_{5})_m$. Since $S(m)$ is infinite by Theorem \ref{Serre's theorem v2}, choose $l\in S(m)$ such that $l>s$, then
$$b_5\left(\frac{ml^rs-1}{6}\right)\equiv(r+1)b_5\left(\frac{ms-1}{6}\right)\equiv(r+1)e\ (\mathrm{mod}\ m).$$

\noindent
Now we fix $l$, choose $\rho\in S(m)$ such that $\rho>l$, then
\begin{equation}
    \label{rho}
    b_5\left(\frac{m\rho n-1}{6}\right)\equiv2b_5\left(\frac{mn-1}{6}\right)\ (\mathrm{mod}\ m)
\end{equation}

\noindent
satisfied for each $n$ coprime to $\rho$. For each $i=1,2,\cdots,m-1$, let $r_i\equiv i(2e)^{-1}-1\ (\mathrm{mod}\ m)$ and $r_i>0$. Let $n=l^{r_i}s$ in (\ref{rho}), we obtain
$$b_5\left(\frac{m\rho l^{r_i}s-1}{6}\right)\equiv2b_5\left(\frac{ml^{r_i}s-1}{6}\right)\equiv2(r_i+1)e\equiv i\ (\mathrm{mod}\ m).$$

Since the variables except $\rho$ are fixed, it suffices to prove that the estimate of the choices of $\rho\gg X/\log X$ and which is derived from Theorem \ref{Serre's theorem v2} and the Prime Number Theorem.

The upper bound $10(m-1)$ of $k$ is obtained by Sturm's Theorem.

\end{proof}

\section{Examples of Ramanujan-type congruences}

\label{examples}

Using Sturm's Theorem, we compute that
$$\modulo{\sum_{n=0}^\infty b_5\bracket{\frac{mn-1}{6}}q^{n}\ |\ T(l)}{0}{m}$$

\noindent
satisfied for $(m,l)=(7,17)$, $(11,41)$, $(13,16519)$. Some elementary computation yields that

\begin{examples}
\begin{equation}
    \notag
    \modulo{b_5(2023n+99)}{0}{7},
\end{equation}
\begin{equation}
    \notag
    \modulo{b_5(18491n+75)}{0}{11},
\end{equation}
\begin{equation}
    \notag
    \modulo{b_5(3547405693n+35791)}{0}{13}.
\end{equation}
\end{examples}

\noindent
Moreover, the congruence $\modulo{b_5(5n+4)}{0}{5}$ implies that
$$\modulo{\sum_{n=0}^\infty b_5\bracket{\frac{5n-1}{6}}q^{n}\ |\ T(l)}{0}{5}$$

\noindent
satisfied for each prime $l$.

\section{Open problems}

We have the following conjecture of the existence of Ramanujan-type congruence.
\begin{conjecture}
Let $m$ be a positive integers, then there are infinitely many Ramanujan-type congruences modulo $m$.
\end{conjecture}

We also have the following conjecture analogous to Newman's Conjecture for the usual partition function $p(n)$.
\begin{conjecture}
\label{Newman's conjecture for b_5(n)}
Let $m$ be a positive integers, then for each integer $i$, there are infinitely many $n$ for which
$$\modulo{b_5(n)}{i}{m}.$$
\end{conjecture}

\begin{remark}
By Theorem \ref{other residue classes}, we verify that Conjecture \ref{Newman's conjecture for b_5(n)} is true for prime $7\leq m\leq 40$.
\end{remark}



\end{document}